\documentclass[reqno,a4paper]{amsart}
\usepackage{amssymb,setspace,graphicx}
\usepackage{ifpdf}
\ifpdf
 \usepackage[hyperindex]{hyperref}
\else
 \expandafter\ifx\csname dvipdfm\endcsname\relax
 \usepackage[hypertex,hyperindex]{hyperref}
 \else
 \usepackage[dvipdfm,hyperindex]{hyperref}
 \fi
\fi
\allowdisplaybreaks[4]
\numberwithin{equation}{section}
\DeclareMathOperator{\td}{d\mspace{-2mu}}
\theoremstyle{plain}
\newtheorem{thm}{Theorem}[section]
\newtheorem{lem}{Lemma}[section]
\newtheorem{cor}{Corollary}[section]

\theoremstyle{remark}
\newtheorem{rem}{Remark}[section]
\theoremstyle{definition}
\newtheorem{dfn}{Definition}[section]
\newcommand{\N}{\mathbb{N}}
\newcommand{\abs}[1]{\left\vert#1\right\vert}

\begin{document}

\title[Integral representations of geometric and logarithmic means]
{Integral representations of the weighted geometric mean and the logarithmic mean}

\author[F. Qi]{Feng Qi}
\address[Feng Qi]{Department of Mathematics, School of Science, Tianjin Polytechnic University, Tianjin City, 300387, China}
\email{\href{mailto: F. Qi <qifeng618@gmail.com>}{qifeng618@gmail.com}, \href{mailto: F. Qi <qifeng618@hotmail.com>}{qifeng618@hotmail.com}, \href{mailto: F. Qi <qifeng618@qq.com>}{qifeng618@qq.com}}
\urladdr{\url{http://qifeng618.wordpress.com}}

\author[X.-J. Zhang]{Xiao-Jing Zhang}
\address[Xiao-Jing Zhang]{Department of Mathematics, School of Science, Tianjin Polytechnic University, Tianjin City, 300387, China}
\email{\href{mailto: X.-J. Zhang <xiao.jing.zhang@qq.com>}{xiao.jing.zhang@qq.com}}

\author[W.-H. Li]{Wen-Hui Li}
\address[Wen-Hui Li]{Department of Mathematics, School of Science, Tianjin Polytechnic University, Tianjin City, 300387, China}
\email{\href{mailto: W.-H. Li <wen.hui.li@foxmail.com>}{wen.hui.li@foxmail.com}}

\begin{abstract}
In the paper, the authors show that the weighted geometric mean and the logarithmic mean are Bernstein functions and establish integral representations of these means by Cauchy's integral theorem in the theory of complex functions.
\end{abstract}

\keywords{Integral representation; Bernstein function; Weighted geometric mean; Logarithmic mean; Induction; Cauchy's integral theorem; Completely monotonic function; Logarithmically completely monotonic function; Stieltjes transform}

\subjclass[2010]{Primary 26E60, 30E20; Secondary 26A48, 44A20}

\thanks{This paper was typeset using \AmS-\LaTeX}

\maketitle

\section{Introduction}

\subsection{Some definitions}
We recall some notions and definitions.

\begin{dfn}[\cite{mpf-1993, widder}]
A function $f$ is said to be completely monotonic on an interval $I$ if $f$ has derivatives of all orders on $I$ and
\begin{equation}
(-1)^{n}f^{(n)}(t)\ge0
\end{equation}
for $x \in I$ and $n \ge0$.
\end{dfn}

\begin{dfn}[\cite{Atanassov}]
If $f^{(k)}(t)$ for some nonnegative integer $k$ is completely monotonic on an interval $I$, but $f^{(k-1)}(t)$ is not completely monotonic on $I$, then $f(t)$ is called a completely monotonic function of $k$-th order on an interval $I$.
\end{dfn}

\begin{dfn}[\cite{compmon2, minus-one}]
A function $f$ is said to be logarithmically completely monotonic on an interval $I$ if its logarithm $\ln f$ satisfies
\begin{equation}
(-1)^k[\ln f(t)]^{(k)}\ge0
\end{equation}
for $k\in\mathbb{N}$ on $I$.
\end{dfn}

\begin{dfn}[\cite{Schilling-Song-Vondracek-2010, widder}]
A function $f:I\subseteq(-\infty,\infty)\to[0,\infty)$ is called a Bernstein function on $I$ if $f(t)$ has derivatives of all orders and $f'(t)$ is completely monotonic on $I$.
\end{dfn}

\begin{dfn}[{\cite[p.~19, Definition~2.1]{Schilling-Song-Vondracek-2010}}]
A Stieltjes function is a function $f:(0,\infty)\to[0,\infty)$ which can be written in the form
\begin{equation}\label{dfn-stieltjes}
f(x)=\frac{a}x+b+\int_0^\infty\frac1{s+x}{\td\mu(s)},
\end{equation}
where $a,b\ge0$ are nonnegative constants and $\mu$ is a nonnegative measure on $(0,\infty)$ such that
\begin{equation*}
\int_0^\infty\frac1{1+s}\td\mu(s)<\infty.
\end{equation*}
\end{dfn}

In the newly-published paper~\cite{psi-proper-fraction-degree-two.tex}, a new notion ``completely monotonic degree'' of functions on $(0,\infty)$ was naturally introduced and initially studied.
\par
It has been proved in~\cite{CBerg, absolute-mon-simp.tex, compmon2, minus-one} that a logarithmically completely monotonic function on an interval $I$ must be completely monotonic on $I$. The set of logarithmically completely monotonic functions on $(0,\infty)$ contains all Stieltjes functions, see~\cite{CBerg} or~\cite[Remark~4.8]{Open-TJM-2003-Banach.tex}.
\par
It is obvious that any nonnegative completely monotonic function of first order is a Bernstein function.
\par
Bernstein functions can be characterized by~\cite[p.~15, Theorem~3.2]{Schilling-Song-Vondracek-2010} which states that a function $f:(0,\infty)\to\mathbb{R}$ is a Bernstein function if and only if it admits the representation
\begin{equation}
f(x)=a+bx+\int_0^\infty\bigl(1-e^{-xt}\bigr)\td\mu(t),
\end{equation}
where $a,b\ge0$ and $\mu$ is a measure on $(0,\infty)$ satisfying
\begin{equation*}
\int_0^\infty\min\{1,t\}\td\mu(t)<\infty.
\end{equation*}
\par
The relation between Bernstein functions and logarithmically completely monotonic functions was discovered in~\cite[pp.~161\nobreakdash--162, Theorem~3]{Chen-Qi-Srivastava-09.tex} and~\cite[p.~45, Proposition~5.17]{Schilling-Song-Vondracek-2010}, which reads that the reciprocal of any Bernstein function is logarithmically completely monotonic.

\subsection{Some means}

We also recall that the extended mean value $E(r,s;x,y)$ may be defined as
\begin{align}
E(r,s;x,y)&=\biggl[\frac{r(y^s-x^s)} {s(y^r-x^r)}\biggr]^{{1/(s-r)}}, & rs(r-s)(x-y)&\ne 0; \\
E(r,0;x,y)&=\biggl[\frac{y^r-x^r}{r(\ln y-\ln x)}\biggr]^{{1/r}}, & r(x-y)&\ne 0; \\
E(r,r;x,y)&=\frac1{e^{1/r}}\biggl(\frac{x^{x^r}}{y^{y^r}}\biggr)^{ {1/(x^r-y^r)}},& r(x-y)&\ne 0; \\
E(0,0;x,y)&=\sqrt{xy}\,, & x&\ne y; \\
E(r,s;x,x)&=x, & x&=y;\notag
\end{align}
where $x$ and $y$ are positive numbers and $r,s\in\mathbb{R}$. Because this mean was first defined in~\cite{stolarsky-mean}, so it is also called Stolarsky's mean. Many special means with two positive variables are special cases of $E$, for example,
\begin{align*}
E(r,2r;x,y)&=M_r(x,y), &&(\text{power mean}) \\
E(1,p;x,y)&=L_p(x,y), &&(\text{generalized logarithmic mean}) \\
E(1,1;x,y)&=I(x,y), &&(\text{exponential mean}) \\
E(1,2;x,y)&=A(x,y), &&(\text{arithmetic mean}) \\
E(0,0;x,y)&=G(x,y), &&(\text{geometric mean}) \\
E(-2,-1;x,y)&=H(x,y), &&(\text{harmonic mean}) \\
E(0,1;x,y)&=L(x,y). &&(\text{logarithmic mean})
\end{align*}
For more information on $E$, please refer to the monograph~\cite{bullenmean}, the papers~\cite{emv-log-convex-simple.tex, Guo-Qi-Filomat-2011-May-12.tex, Qi-Springer-2012-Srivastava.tex}, and a lot of closely-related references therein.

\subsection{The arithmetic mean is a Bernstein function}

It is easy to see that the arithmetic mean
$$
A_{x,y}(t)=A(x+t,y+t)=A(x,y)+t
$$
is a trivial Bernstein function of $t\in(-\min\{x,y\},\infty)$ for $x,y>0$.

\subsection{The harmonic mean is a Bernstein function}

In~\cite{Geomeric-Mean-Complete-revised.tex}, the harmonic mean
\begin{equation}
H_{x,y}(t)=H(x+t,y+t)=\frac2{\frac1{x+t}+\frac1{y+t}}
\end{equation}
for $t\in(-\min\{x,y\},\infty)$ and $x,y>0$ with $x\ne y$ was proved to be a Bernstein function and
\begin{gather}\label{H-int-reprent}
H_{x,y}(s)=H(x,y)+s+\frac{(x-y)^2}4 \int_0^\infty\bigl(1-e^{-su}\bigr) e^{-(x+y)u/2}\td u,\\
\label{H-int-reprent-A=H}
H(x,y)=A(x,y)-\frac{(x-y)^2}2 \int_0^\infty e^{-(x+y)u}\td u,\\
H(s,y+s)=s+\frac{y^2}4 \int_0^\infty\bigl(1-e^{-su}\bigr) e^{-yu/2}\td u.
\end{gather}

\subsection{The exponential mean is a Bernstein function}

In~\cite[p.~116, Remark~6]{new-upper-kershaw-JCAM.tex}, it was pointed out that the reciprocal of the exponential mean
\begin{equation}
I_{x,y}(t)=I(x+t,y+t)=\frac1e\biggl[\frac{(x+t)^{x+t}}{(y+t)^{y+t}}\biggr]^{ {1/(x-y)}}
\end{equation}
for $x,y>0$ with $x\ne y$ is a logarithmically completely monotonic function of $t\in(-\min\{x,y\},\infty)$ and that, by using
\begin{equation}
I(x,y)=\exp\biggl(\frac1{y-x}\int_x^y\ln u\td u\biggr),
\end{equation}
the exponential mean $I_{x,y}(t)$ for $t>-\min\{x,y\}$ with $x\ne y$ is also a completely monotonic function of first order (that is, a Bernstein function).

\subsection{The logarithmic mean is a Bernstein function}

In~\cite[p.~616]{gamma-psi-batir.tex}, it was concluded that the logarithmic mean
\begin{equation}\label{log-mean+t-eq}
L_{x,y}(t)=L(x+t,y+t)
\end{equation}
is increasing and concave in $t>-\min\{x,y\}$ for $x,y>0$ with $x\ne y$. More strongly, it was proved in~\cite[Theorem~1]{log-mean-comp-mon.tex-mia} that the logarithmic mean $L_{x,y}(t)$ for $x,y>0$ with $x\ne y$ is a completely monotonic function of first order in $t\in(-\min\{x,y\},\infty)$, that is, the logarithmic mean $L_{x,y}(t)$ is a Bernstein function of $t\in(-\min\{x,y\},\infty)$. The proof of~\cite[Theorem~1]{log-mean-comp-mon.tex-mia} is based on making use of the integral representation
\begin{equation}\label{Log-Mean-int-first}
L(x,y)=\int_0^1x^{u}y^{1-u}\td u
\end{equation}
and proving that the weighted geometric mean
\begin{equation}\label{weight-geometric=t}
G_{x,y;\lambda}(t)=(x+t)^{\lambda}(y+t)^{1-\lambda}
\end{equation}
for $\lambda\in(0,1)$ and $x,y\in\mathbb{R}$ with $x\ne y$ is a Bernstein function of $t>-\min\{x,y\}$.

\subsection{The geometric mean is a Bernstein function}

After the weighted geometric mean $G_{x,y;\lambda}(t)$ was proved in~\cite{log-mean-comp-mon.tex-mia} to be a Bernstein function, the statement that the geometric mean $G_{x,y;1/2}(t)$ is a Bernstein function was recently recovered in~\cite{Geomeric-Mean-Complete-revised.tex} by several approaches. More importantly, the integral representation
\begin{equation}\label{G(x-y)-integ-exp}
G_{x,y;1/2}(z)=G(x,y)+z+\frac{(x-y)^2}{2\pi}\int_0^\infty\frac{\rho((x-y)s)}s e^{-ys}\bigl(1-e^{-sz}\bigr)\td s
\end{equation}
for $x>y>0$ and $z\in\mathbb{C}\setminus(-\infty,-y]$ was established in~\cite{Geomeric-Mean-Complete-revised.tex}, where
\begin{equation}\label{rho-funct-dfn}
\begin{aligned}
\rho(s)&=\int_0^{1/2} q(u)\bigl[1-e^{-(1-2u)s}\bigr]e^{-us}\td u\\
&=\int_0^{1/2} q\biggl(\frac12-u\biggr)\bigl(e^{us}-e^{-us}\bigr)e^{-s/2}\td u\\
&\ge0
\end{aligned}
\end{equation}
on $(0,\infty)$ and
\begin{equation}
q(u)=\sqrt{\frac1u-1}\,-\frac1{\sqrt{1/u-1}\,}
\end{equation}
on $(0,1)$.
\par
Let $0<a_k\le a_{k+1}$ for $1\le k\le n-1$ and $a+z=(a_1+z,a_2+z,\dotsc,a_n+z)$ for $z\in\mathbb{C}\setminus(-\infty,-a_1]$. Then the geometric mean
$$
G_n(a+z)=\sqrt[n]{\prod_{\ell=1}^n(a_\ell+z)}
$$
has the integral representation
\begin{equation}\label{AG-New-eq1}
G_n(a+z)=A_n(a)+z-\frac1\pi\sum_{\ell=1}^{n-1}\sin\frac{\ell\pi}n \int_{a_\ell}^{a_{\ell+1}} \Biggl|\prod_{k=1}^n(a_k-t)\Biggr|^{1/n} \frac{\td t}{t+z}.
\end{equation}
This conclusion was gained in~\cite{AG-Ineq-New-Proof.tex}. From this, it is easy to see that the geometric mean $G_n(a+t)$ is a Bernstein function of $t\in(-a_1,\infty)$. More interestingly, the well-known inequality between $A_n(a)$ and $G_n(a)$ can be derived from~\eqref{AG-New-eq1}.

\subsection{Main results of this paper}

In this paper, we will find an integral representation of the weighted geometric mean $G_{x,y;\lambda}(t)$ defined by~\eqref{weight-geometric=t}, and then substitute it into~\eqref{Log-Mean-int-first} to obtain a new integral representation of the logarithmic mean.

\section{Lemmas}

For proving our main results, we need the following lemmas.

\begin{lem}\label{Test-thm-h(t)}
For $t>0$ and $\alpha\in(-1,1)$, let
\begin{equation}
h_\alpha(t)=\biggl(1+\frac1t\biggr)^\alpha.
\end{equation}
Then the derivatives of $h_\alpha(t)$ can be computed by
\begin{equation}\label{h-alpha=m-deriv}
h_\alpha^{(m)}(t)=\frac{(-1)^m}{t^{m}(1+t)^{m}}\biggl(1+\frac1t\biggr)^{\alpha} \sum_{k=0}^{m-1}{a_{\alpha,m,k}t^k},
\end{equation}
where $m\in\N$ and
\begin{equation}\label{a-alpha-m-k}
a_{\alpha,m,k}=k!\binom{m}{k}\binom{m-1}{k} \prod_{\ell=0}^{m-k-1}(\alpha+\ell).
\end{equation}
Consequently,
\begin{enumerate}
\item
if $\alpha\in(0,1)$, the function $h_\alpha(t)$ is completely monotonic on $(0,\infty)$;
\item
if $\alpha\in(-1,0)$, the function $h_\alpha(t)$ is a Bernstein function on $(0,\infty)$;
\item
the derivatives of the function
\begin{equation}
H_\alpha(t)=\frac{h_\alpha(t)}\alpha-\frac{h_{\alpha-1}(t)}{\alpha-1}
\end{equation}
may be calculated by
\begin{equation}
H_\alpha^{(m)}(t)=\frac{(-1)^m}{t^{m}(1+t)^{m+1}}\biggl(1+\frac1t\biggr)^{\alpha} \sum_{k=0}^{m-1}{b_{\alpha,m,k}t^k},
\end{equation}
where $m\in\N$ and
\begin{equation}
b_{\alpha,m,k}=k!\binom{m+1}{k}\binom{m-1}{k} \prod_{\ell=1}^{m-k-1}(\alpha+\ell);
\end{equation}
\item
the function $H_\alpha(t)$ is completely monotonic for all $\alpha\in(0,1)$ on $(0,\infty)$.
\end{enumerate}
\end{lem}

\begin{proof}
It is easy to see that
\begin{equation*}
h_\alpha'(t)=-\alpha\biggl(1+\frac1t\biggr)^{\alpha-1}\frac1{t^2} =-\alpha\biggl(1+\frac1t\biggr)^\alpha\frac1{t(1+t)}.
\end{equation*}
This means that the formulas~\eqref{h-alpha=m-deriv} and~\eqref{a-alpha-m-k} are valid for $m=1$.
\par
Assume that the formulas~\eqref{h-alpha=m-deriv} and~\eqref{a-alpha-m-k} are valid for some $m>1$. By this inductive hypothesis, a simple calculation gives
\begin{multline*}
h_\alpha^{(m+1)}(t)=\bigl[h_\alpha^{(m)}(t)\bigr]'
=\Biggl[\frac{(-1)^m}{t^{m}(1+t)^{m}}\biggl(1+\frac1t\biggr)^{\alpha} \sum_{k=0}^{m-1}{a_{\alpha,m,k}t^k}\Biggr]'\\
\begin{aligned}
&=(-1)^{m+1}\sum_{k=0}^{m-1}\frac{a_{\alpha,m,k}} {[t^{m+\alpha-k}(1+t)^{m-\alpha}]^2}\\
&\quad\times\bigl[(m+\alpha-k)t^{m+\alpha-k-1}(1+t)^{m-\alpha}  +(m-\alpha)t^{m+\alpha-k}(1+t)^{m-\alpha-1}\bigr]\\
&=\frac{(-1)^{m+1}}{t^{m+1}(1+t)^{m+1}}\biggl(1+\frac1t\biggr)^\alpha \sum_{k=0}^{m-1}a_{\alpha,m,k}[m+\alpha-k+(2m-k)t]t^k\\ &=\frac{(-1)^{m+1}}{t^{m+1}(1+t)^{m+1}}\biggl(1+\frac1t\biggr)^\alpha \Biggl\{(m+\alpha)a_{\alpha,m,0}+(m+1)a_{\alpha,m,m-1}t^m\\
&\quad +\sum_{k=1}^{m-1}[(m+\alpha-k){a_{\alpha,m,k}}+(2m-k+1){a_{\alpha,m,k-1}}]t^k\Biggr\}\\
&=\frac{(-1)^{m+1}}{t^{m+1}(1+t)^{m+1}}\biggl(1+\frac1t\biggr)^\alpha \sum_{k=0}^{m}{a_{\alpha,m+1,k}t^k}.
\end{aligned}
\end{multline*}
This shows that the formulas~\eqref{h-alpha=m-deriv} and~\eqref{a-alpha-m-k} are valid for all $m\ge1$.
\par
The left proofs are straightforward, so we omit them. The proof of Lemma~\ref{Test-thm-h(t)} is completed.
\end{proof}

\begin{lem}\label{weig-geom-lem2}
For $\alpha\in(-1,1)$ and $z\in\mathbb{C}\setminus(-\infty,0]$, the principal branch of the function $h_\alpha(z)$ has the integral representation
\begin{equation}\label{weighted-geometric-eq1}
h_\alpha(z)=1+\frac{\sin(\alpha\pi)}{\pi} \int_0^1\biggl(\frac1u-1\biggr)^\alpha\frac{\td u}{u+z}.
\end{equation}
Consequently, the function $h_\alpha(t)$ is a Stieltjes function for $\alpha\in(0,1)$ and a Bernstein function for $\alpha\in(-1,0)$ on $(0,\infty)$.
\end{lem}

\begin{proof}
By standard arguments, we can obtain immediately that
\begin{gather}\label{weighted-geometric-eq2}
\lim_{z\to0}[zh_\alpha(z)]=\lim_{z\to0}\biggl[z\biggl(1+\frac1z\biggr)^\alpha\biggr] =\lim_{z\to0}\bigl[z^{1-\alpha}(1+z)^\alpha\bigr]=0,\\
\lim_{z\to\infty}h_\alpha(z)=\lim_{z\to\infty}\biggl(1+\frac1z\biggr)^\alpha=1, \label{weighted-geometric-eq3}
\end{gather}
and
\begin{equation}\label{weighted-geometric-eq4}
h_\alpha(\overline{z})=\overline{h_\alpha(z)}.
\end{equation}
\par
For $t\in(0,\infty)$ and $\varepsilon>0$, we have
\begin{align*}
h_\alpha(-t+i\varepsilon)&=\biggl(1+\frac1{-t+i\varepsilon}\biggr)^\alpha\\
&=\biggl(\frac{1-t+i\varepsilon}{-t+i\varepsilon}\biggr)^\alpha\\
&=\biggl[\frac{(1-t+i\varepsilon)(-t-i\varepsilon)}{t^2+\varepsilon^2}\biggr]^\alpha\\
&=\biggl(\frac{t^2-t+\varepsilon^2-i\varepsilon}{t^2+\varepsilon^2}\biggr)^\alpha\\
&=\exp\biggl[\alpha\biggl(\ln\abs{\frac{t^2-t+\varepsilon^2-i\varepsilon}{t^2+\varepsilon^2}} +i\arg\frac{t^2-t+\varepsilon^2-i\varepsilon}{t^2+\varepsilon^2}\biggr)\biggr]\\
&=
\begin{cases}
\exp\biggl[\alpha\biggl(\ln\abs{\dfrac{t^2-t+\varepsilon^2-i\varepsilon}{t^2+\varepsilon^2}} & \\
\hskip2em-i\arctan\dfrac{\varepsilon}{t^2-t+\varepsilon^2}\biggr)\biggr],& t^2-t+\varepsilon^2>0\\
\exp\biggl[\alpha\biggl(\ln\abs{\dfrac{t^2-t+\varepsilon^2-i\varepsilon}{t^2+\varepsilon^2}} & \\
\hskip2em+i\arctan\dfrac{\varepsilon}{t^2-t+\varepsilon^2}-\pi i\biggr)\biggr],& t^2-t+\varepsilon^2<0\\
\exp\biggl[\alpha\biggl(\ln\dfrac{\varepsilon}{t^2+\varepsilon^2} -\dfrac{\pi i}2\biggr)\biggr],& t^2-t+\varepsilon^2=0
\end{cases}\\
&\to
\begin{cases}
\exp\biggl(\alpha\ln\dfrac{t-1}t\biggr), & t>1\\
\exp\biggl[\alpha\biggl(\ln\dfrac{1-t}t-\pi i\biggr)\biggr], & 0<t<1\\
0, & t=1
\end{cases}\\
&=
\begin{cases}
\biggl(\dfrac{t-1}t\biggr)^\alpha, & t>1\\
\biggl(\dfrac{1-t}t\biggr)^\alpha[\cos(\alpha\pi)-\sin(\alpha\pi)], & 0<t<1\\
0, & t=1
\end{cases}
\end{align*}
as $\varepsilon\to0^+$. As a result,
\begin{equation}\label{weighted-geometric-eq5}
\lim_{\varepsilon\to0^+}\Im f(-t+i\varepsilon)=
\begin{cases}
0,&t\ge1;\\
-\biggl(\dfrac1t-1\biggr)^\alpha\sin(\alpha\pi), &0<t<1.
\end{cases}
\end{equation}
\par
Let $D$ be a bounded domain with piecewise smooth boundary $\partial D$. The famous Cauchy integral formula (see~\cite[p.~113]{Gamelin-book-2001}) reads that if $f(z)$ is holomorphic on $D$ and if $f(z)$ extends smoothly to the boundary $\partial D$ of $D$, then
\begin{equation}\label{cauchy-formula=eq}
f(z)=\frac1{2\pi i}\oint_{\partial D}\frac{f(w)}{w-z}\td w,\quad z\in D.
\end{equation}
\par
For any but fixed point $z\in\mathbb{C}\setminus(-\infty,0]$, choose $0<\varepsilon<1$ and $r>0$ such that $0<\varepsilon<|z|<r$, and consider the positively oriented contour $C(\varepsilon,r)$ in $\mathbb{C}\setminus(-\infty,0]$ consisting of the half circle $z=\varepsilon e^{i\theta}$ for $\theta\in\bigl[-\frac\pi2,\frac\pi2\bigr]$ and the half lines $z=x\pm i\varepsilon$ for $x\le0$ until they cut the circle $|z|=r$, which close the contour at the points $-r(\varepsilon)\pm i\varepsilon$, where $0<r(\varepsilon)\to r$ as $\varepsilon\to0$. See Figure~\ref{note-on-li-chen-conj.eps}.
\begin{figure}[htbp]
\includegraphics[width=0.75\textwidth]{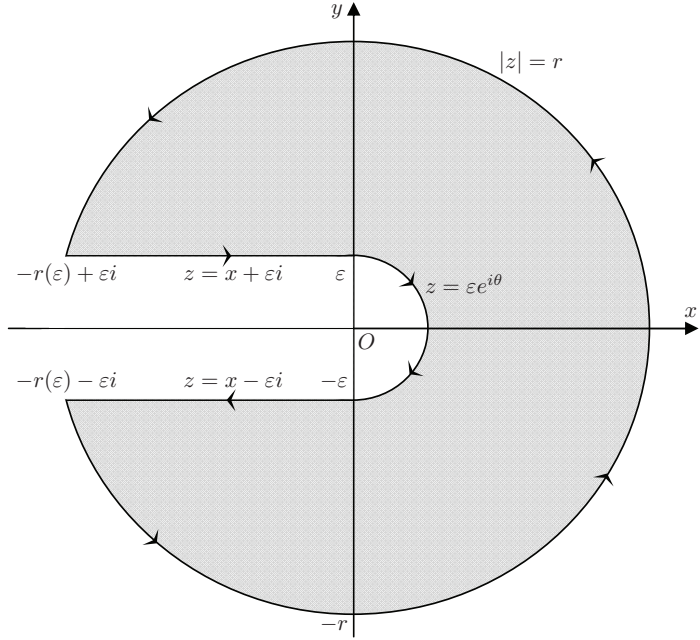}\\
\caption{The contour $C(\varepsilon,r)$}\label{note-on-li-chen-conj.eps}
\end{figure}
\par
By the above mentioned Cauchy integral formula, we have
\begin{equation}
\begin{split}\label{h(z)-Cauchy-Apply}
h_\alpha(z)&=\frac1{2\pi i}\oint_{C(\varepsilon,r)}\frac{h_\alpha(w)}{w-z}\td w\\
&=\frac1{2\pi i}\biggl[\int_{\pi/2}^{-\pi/2}\frac{i\varepsilon e^{i\theta}h\bigl(\varepsilon e^{i\theta}\bigr)}{\varepsilon e^{i\theta}-z}\td\theta +\int_{-r(\varepsilon)}^0 \frac{h_\alpha(x+i\varepsilon)}{x+i\varepsilon-z}\td x \\
&\quad+\int_0^{-r(\varepsilon)}\frac{h_\alpha(x-i\varepsilon)}{x-i\varepsilon-z}\td x +\int_{\arg[-r(\varepsilon)-i\varepsilon]}^{\arg[-r(\varepsilon)+i\varepsilon]}\frac{ir e^{i\theta}h\bigl(re^{i\theta}\bigr)}{re^{i\theta}-z}\td\theta\biggr].
\end{split}
\end{equation}
By the limit~\eqref{weighted-geometric-eq2}, it follows that
\begin{equation}\label{zf(z)=0}
\lim_{\varepsilon\to0^+}\int_{\pi/2}^{-\pi/2}\frac{i\varepsilon e^{i\theta}h_\alpha\bigl(\varepsilon e^{i\theta}\bigr)}{\varepsilon e^{i\theta}-z}\td\theta=0.
\end{equation}
In virtue of the limit~\eqref{weighted-geometric-eq3}, we can derive that
\begin{equation}\label{big-circle-int=0}
\lim_{\substack{\varepsilon\to0^+\\r\to\infty}} \int_{\arg[-r(\varepsilon)-i\varepsilon]}^{\arg[-r(\varepsilon)+i\varepsilon]}\frac{ir e^{i\theta}h_\alpha\bigl(re^{i\theta}\bigr)}{re^{i\theta}-z}\td\theta
=\lim_{r\to\infty}\int_{-\pi}^{\pi}\frac{ir e^{i\theta}h_\alpha\bigl(re^{i\theta}\bigr)}{re^{i\theta}-z}\td\theta=2\pi i.
\end{equation}
Making use of the limits~\eqref{weighted-geometric-eq4} and~\eqref{weighted-geometric-eq5} leads to
\begin{align}
&\quad\int_{-r(\varepsilon)}^0 \frac{h_\alpha(x+i\varepsilon)}{x+i\varepsilon-z}\td x
+\int_0^{-r(\varepsilon)}\frac{h_\alpha(x-i\varepsilon)}{x-i\varepsilon-z}\td x \notag\\
&=\int_{-r(\varepsilon)}^0 \biggl[\frac{h_\alpha(x+i\varepsilon)}{x+i\varepsilon-z}-\frac{h_\alpha(x-i\varepsilon)}{x-i\varepsilon-z}\biggr]\td x\notag\\
&=\int_{-r(\varepsilon)}^0\frac{(x-i\varepsilon-z)h_\alpha(x+i\varepsilon) -(x+i\varepsilon-z)h_\alpha(x-i\varepsilon)} {(x+i\varepsilon-z)(x-i\varepsilon-z)}\td x\notag\\
&=\int_{-r(\varepsilon)}^0\frac{(x-z)[h_\alpha(x+i\varepsilon)-h_\alpha(x-i\varepsilon)] -i\varepsilon[h_\alpha(x-i\varepsilon)+h_\alpha(x+i\varepsilon)]} {(x+i\varepsilon-z)(x-i\varepsilon-z)}\td x\notag\\
&=2i\int_{-r(\varepsilon)}^0\frac{(x-z)\Im h_\alpha(x+i\varepsilon) -\varepsilon\Re h_\alpha(x+i\varepsilon)} {(x+i\varepsilon-z)(x-i\varepsilon-z)}\td x\notag\\
&\to2i\int_{-r}^0\frac{\lim_{\varepsilon\to0^+}\Im h_\alpha(x+i\varepsilon)}{x-z}\td x\notag\\
&=-2i\int^r_0\frac{\lim_{\varepsilon\to0^+}\Im h_\alpha(-t+i\varepsilon)}{t+z}\td t\notag\\
&\to-2i\int^\infty_0\frac{\lim_{\varepsilon\to0^+}\Im h_\alpha(-t+i\varepsilon)}{t+z}\td t\notag\\
&=2i\sin(\alpha\pi)\int_0^1\biggl(\frac1t-1\biggr)^\alpha\frac{\td t}{t+z} \label{level0lines}
\end{align}
as $\varepsilon\to0^+$ and $r\to\infty$. Substituting equations~\eqref{zf(z)=0}, \eqref{big-circle-int=0}, and~\eqref{level0lines} into~\eqref{h(z)-Cauchy-Apply} and simplifying produce the integral representation~\eqref{weighted-geometric-eq1}. The proof of Lemma~\ref{weig-geom-lem2} is completed.
\end{proof}

\section{An integral representation of the weighted geometric mean}

Utilizing lemmas in the above section, we now prove that the weighted geometric mean $G_{x,y;\lambda}(t)$ is a Bernstein function of $t>-\min\{x,y\}$ and present an integral representation of the geometric mean $G_{x,y;\lambda}(z)$ for $z\in\mathbb{C}\setminus(-\infty,-\min\{x,y\}]$.

\begin{thm}\label{weighted-geom=mean-thm}
For $\lambda\in(0,1)$ and $x,y\in\mathbb{R}$ with $x\ne y$, the weighted geometric mean $G_{x,y;\lambda}(t)$ defined by~\eqref{weight-geometric=t} is a Bernstein function of $t>-\min\{x,y\}$.
\end{thm}

\begin{proof}
When $x>y>0$, a direct differentiation yields
\begin{align}
G_{x,y;\lambda}'(t)&=\lambda(1-\lambda)\biggl[\frac1\lambda\biggl(\frac{x+t}{y+t}\biggr)^\lambda +\frac1{1-\lambda}\biggl(\frac{y+t}{x+t}\biggr)^{1-\lambda}\biggr]\notag\\
&=\lambda(1-\lambda)\biggl[\frac1\lambda\biggl(1+\frac{x-y}{y+t}\biggr)^\lambda -\frac1{\lambda-1}\biggl(1+\frac{x-y}{y+t}\biggr)^{\lambda-1}\biggr]\notag\\
&=\lambda(1-\lambda)\biggl[\frac1\lambda h_\lambda\biggl(\frac{y+t}{x-y}\biggr) -\frac1{\lambda-1}h_{\lambda-1}\biggl(\frac{y+t}{x-y}\biggr)\biggr]\notag\\
&=\lambda(1-\lambda)H_{\lambda}\biggl(\frac{y+t}{x-y}\biggr).\label{H-lambda-x-y}
\end{align}
By the complete monotonicity of the function $H_\alpha$ obtained in Lemma~\ref{Test-thm-h(t)}, it is immediate to see that the derivative $G_{x,y;\lambda}'(t)$ is completely monotonic, and so the geometric mean $G_{x,y;\lambda}(t)$ is a Bernstein function for $x>y>0$ and $\lambda\in(0,1)$. Considering the symmetry property
\begin{equation}\label{G-symmetric-lamda}
G_{x,y;\lambda}(t)=G_{y,x;1-\lambda}(t)
\end{equation}
reveals that, no matter $y>x>0$ or $x>y>0$, the geometric mean $G_{x,y;\lambda}(t)$ is a Bernstein function of $t>-\min\{x,y\}$.
\end{proof}

\begin{thm}\label{weight-geometric-int-repres-thm}
For $\lambda\in(0,1)$ and $x>y>0$, the principal branch of the weighted geometric mean $G_{x,y;\lambda}(z)$ defined by~\eqref{weight-geometric=t} has the integral representation
\begin{multline}\label{weight-geometric-int-repres-eq}
G_{x,y;\lambda}(z)=x^{\lambda}y^{1-\lambda}+z \\
+\frac{\sin(\lambda\pi)}{\pi} (x-y)\int_0^\infty  \frac{F(\lambda,(x-y)s)}s e^{-sy}(1-e^{-sz})\td s,
\end{multline}
where $z\in\mathbb{C}\setminus(-\infty,-y]$ and
\begin{equation}\label{F(lambda-s)}
F(\lambda,s)=\int_0^1\biggl(\frac1u-1\biggr)^\lambda \biggl(1-\frac{\lambda}{1-u}\biggr)e^{-us}\td u>0.
\end{equation}
Consequently, the geometric mean $G_{x,y;\lambda}(t)$ for $\lambda\in(0,1)$ and $x,y>0$ is a Bernstein function of $t>-\min\{x,y\}$.
\end{thm}

\begin{proof}
By the integral representation~\eqref{weighted-geometric-eq1} in Lemma~\ref{weig-geom-lem2}, we have
\begin{equation*}
H_{\lambda}\biggl(\frac{y+z}{x-y}\biggr)
=\frac1{\lambda(1-\lambda)}+ \int_0^1Q_{\lambda}(u)\frac{\td u}{u+\frac{y+z}{x-y}},
\end{equation*}
where $H_{\lambda}$ is defined by~\eqref{H-lambda-x-y} and
\begin{align*}
Q_{\lambda}(u)&=\frac{\sin(\lambda\pi)}{\lambda\pi}\biggl(\frac1u-1\biggr)^\lambda -\frac{\sin[(\lambda-1)\pi]}{(\lambda-1)\pi}\biggl(\frac1u-1\biggr)^{\lambda-1}\\
&=\frac{\sin(\lambda\pi)}{\pi}\biggl(\frac1u-1\biggr)^\lambda\biggl[\frac1\lambda -\frac{1}{1-\lambda}\biggl(\frac1u-1\biggr)^{-1}\biggr]\\
&=\frac{\sin(\lambda\pi)}{\lambda(1-\lambda)\pi}\biggl(\frac1u-1\biggr)^\lambda \biggl(1-\frac{\lambda}{1-u}\biggr).
\end{align*}
Accordingly,
\begin{align*}
G_{x,y;\lambda}'(z)&=1+ \lambda(1-\lambda)\int_0^1Q_{\lambda}(u)\frac{\td u}{u+\frac{y+z}{x-y}}\\
&=1+ \frac{\sin(\lambda\pi)}{\pi} \int_0^\infty F(\lambda,s) \exp\biggl(-\frac{y+z}{x-y}s\biggr)\td s.
\end{align*}
Integrating on both sides of this equation with respect to $z$ from $0$ to $w$ gives
\begin{align*}
G_{x,y;\lambda}(w)&=G_{x,y;\lambda}(0)+w\\
&\quad+\frac{\sin(\lambda\pi)}{\pi} \int_0^\infty F(\lambda,s) \biggl[\int_0^w\exp\biggl(-\frac{y+z}{x-y}s\biggr)\td z\biggr] \td s\\
&=x^{\lambda}y^{1-\lambda}+w+\frac{\sin(\lambda\pi)}{\pi} (x-y)\\
&\quad\times\int_0^\infty  \frac{F(\lambda,s)}s\exp\biggl(-\frac{sy}{x-y}\biggr) \biggl[1-\exp\biggl(-\frac{sw}{x-y}\biggr)\biggr]\td s.
\end{align*}
Changing the variable $s$ by $(x-y)s$ and replacing $w$ by $z$ in the above integral lead to the integral representation~\eqref{weight-geometric-int-repres-eq}.
\par
It is obvious that
\begin{align*}
F(\lambda,s)&=\biggl(\int_0^{1-\lambda}+\int_{1-\lambda}^1\biggr)\biggl(\frac1u-1\biggr)^\lambda \biggl(1-\frac{\lambda}{1-u}\biggr)e^{-us}\td u\\
&>e^{-(1-\lambda)s}\biggl(\int_0^{1-\lambda}+\int_{1-\lambda}^1\biggr)\biggl(\frac1u-1\biggr)^\lambda \biggl(1-\frac{\lambda}{1-u}\biggr)\td u\\
&=e^{-(1-\lambda)s}\int_0^1\biggl(\frac1u-1\biggr)^\lambda \biggl(1-\frac{\lambda}{1-u}\biggr)\td u\\
&=0.
\end{align*}
The proof of Theorem~\ref{weight-geometric-int-repres-thm} is completed.
\end{proof}

\begin{cor}\label{weight-AG-cor}
For $\lambda\in(0,1)$ and $x>y>0$, the difference between the weighted arithmetic and geometric means has the following integral representation
\begin{equation}\label{weight-AG-cor=eq}
\bigl[\lambda x+(1-\lambda)y\bigr]-x^{\lambda}y^{1-\lambda} =\frac{\sin(\lambda\pi)}{\pi} (x-y)\int_0^\infty  \frac{F(\lambda,(x-y)s)}s e^{-sy}\td s,
\end{equation}
where $F(\lambda,s)$ is defined by~\eqref{F(lambda-s)}. Consequently, the weighted arithmetic mean of two positive numbers is not less than the weighted geometric mean of two positive numbers.
\end{cor}

\begin{proof}
This follows from taking $z\to\infty$ on both sides of~\eqref{weight-geometric-int-repres-eq} and making use of the fact that
\begin{align*}
\lim_{z\to\infty}[G_{x,y;\lambda}(z)-z]
&=\lim_{z\to\infty}\biggl\{z\biggl[\biggl(1+\frac{x}z\biggr)^\lambda \biggl(1+\frac{y}z\biggr)^{1-\lambda}-1\biggr]\biggr\} \\
&=\lim_{z\to0}\frac{(1+{x}z)^\lambda(1+{y}z)^{1-\lambda}-1}z \\
&=\lambda x+(1-\lambda)y.
\end{align*}
Corollary~\ref{weight-AG-cor} is thus proved.
\end{proof}

\section{An integral representation of the logarithmic mean}

Employing the integral representation~\eqref{weight-geometric-int-repres-eq} in Theorem~\ref{weight-geometric-int-repres-thm}, we now derive an integral representation of the logarithmic mean $L_{x,y}(z)$ for $z\in\mathbb{C}\setminus(-\infty,-\min\{x,y\}]$.

\begin{thm}\label{log=mean-thm}
For $x>y>0$, the logarithmic mean $L_{x,y}(z)$ defined by~\eqref{log-mean+t-eq} has the integral representation
\begin{equation}\label{log-mean-int-express}
L_{x,y}(z)=L(x,y)+z +\frac{x-y}{\pi}\int_0^\infty \frac{P_{x,y}(s)}s e^{-sy}\bigl(1-e^{-sz}\bigr)\td s
\end{equation}
for $z\in\mathbb{C}\setminus(-\infty,-y]$, where
\begin{equation}\label{P(x,y,s)-dfn}
P_{x,y}(s)=\int_0^1\sin(\lambda\pi)F(\lambda,(x-y)s)\td\lambda
\end{equation}
and the function $F$ is defined by~\eqref{F(lambda-s)}.
Consequently, the logarithmic mean $L_{x,y}(t)$ is a Bernstein function of $t>-\min\{x,y\}$.
\end{thm}

\begin{proof}
This may be deduced from integrating with respect to $\lambda$ from $0$ to $1$ on both sides of~\eqref{weight-geometric-int-repres-eq} and considering~\eqref{Log-Mean-int-first}.
\end{proof}

\begin{cor}\label{A-L-diff-cor}
For $x>y>0$, the difference between the arithmetic and logarithmic means satisfies
\begin{equation}
A(x,y)-L(x,y)=\frac{x-y}{\pi}\int_0^\infty \frac{P_{x,y}(s)}s e^{-sy}\td s,
\end{equation}
where the function $P_{x,y}(s)$ is defined by~\eqref{P(x,y,s)-dfn}.
\end{cor}

\begin{proof}
This may be derived from letting $z\to\infty$ on both sides of~\eqref{log-mean-int-express} and using the fact that $\lim_{z\to\infty}[L_{x,y}(z)-z]=A(x,y)$. Corollary~\ref{A-L-diff-cor} is thus proved.
\end{proof}

\begin{rem}
The multi-variable analogy of the integral representation~\eqref{weight-geometric-int-repres-eq} has been established in~\cite{AG-Ineq-New-Proof.tex} and has been employed in~\cite{1st-Sirling-Number-2012.tex} to find an integral representation of Stirling numbers of the first kind. Some results in this paper were included in the thesis~\cite{Zhang-Xiao-Jing-Thesis.tex}.
\end{rem}

\begin{rem}
After this paper was completed, we found the paper~\cite{A.B-CM-Mean} which is motivated by the paper~\cite{log-mean-comp-mon.tex-mia} and discovered some conditions on $r$ and $s$ such that Stolarsky's $E(r,s;x+t,y+t)$ are Bernstein functions of $t$.
\end{rem}

\end{document}